\documentclass[12pt]{article} 
\usepackage{amsmath}
\usepackage{amssymb}
\usepackage{amsthm}
\usepackage{mathtools}
\DeclarePairedDelimiter\ceil{\lceil}{\rceil}

\usepackage[utf8]{inputenc}
\usepackage[english]{babel}
\usepackage{color}
\usepackage{hyperref}

\newcommand{\arr}{\mathbb{R}}

\newcommand{\Mod}[1]{\ (\mathrm{mod}\ #1)}

\usepackage{wrapfig}

\newtheorem{theorem}{Theorem}[section]
\newtheorem{corollary}[theorem]{Corollary}

\newtheorem{lemma}[theorem]{Lemma}

\newtheorem{proposition}[theorem]{Proposition}
\newtheorem{remark}[theorem]{Remark}

\title{Dot product chains}
\date{\today}
\author{Shelby Kilmer, Caleb Marshall, and Steven Senger}

\begin{document}
\maketitle
\begin{abstract}
We study a variant of Erd\H os' unit distance problem, concerning dot products between successive pairs of points chosen from a large finite point set. Specifically, given a large finite set of $n$ points $E$, and a sequence of nonzero dot products $(\alpha_1,\ldots,\alpha_k)$, we give upper and lower bounds on the maximum possible number of tuples of distinct points $(A_1,\dots, A_{k+1})\in E^{k+1}$ satisfying $A_j \cdot A_{j+1}=\alpha_j$ for every $1\leq j \leq k$.
\end{abstract}

\section{Introduction}
In \cite{Erd46}, Erd\H os introduced two popular problems in discrete geometry, the single distance problem and the distinct distances problem. Given a finite point set in the plane, the single distance problem asks how often a single distance can occur between pairs of points, while the distinct distances problem asks how many distinct distances must be determined by pairs of points. In the decades since these problems were first posed, they have been studied by many people, with varying degrees of success. See \cite{BMP, GIS} for surveys of these and related problems. The distinct distances problem was resolved in 2010 by Guth and Katz in \cite{GK}. One popular variant of this family of problems involves replacing the distance between two points by the dot product of two points, \cite{GIS, Steinerberger}.

In addition to studying the dot products determined by pairs of points chosen from a set, there has been much interest in studying dot products determined by larger subsets of points. See \cite{Bahls, Fickus, IS, TV12} for some examples and applications. In this note, we concern ourselves with chains, which are sequences of points restricted by the dot products between successive pairs. We borrow notation from related problems on distances in \cite{BIT, PSS}. Specifically, if we fix a $k$-tuple of real numbers, $(\alpha_1, \alpha_2, \dots, \alpha_k)$, then a {\it $k$-chain} of that type is a $(k+1)$-tuple of points, $(R_1, R_2, \dots, R_{k+1}),$ such that for all $j=1,\dots, k,$ we have $R_j\cdot R_{j+1}=\alpha_j.$ For example, if we fix a triple of real numbers, $(\alpha, \beta, \gamma)$, a 3-chain of that type will be a set of four points, where the dot product of the first two points is $\alpha$, the dot product of the middle two points is $\beta$, and the dot product of the last two points is $\gamma.$ We follow convention and refer to 2-chains as hinges.

We will assume a given $k$-tuple $(\alpha_1, \alpha_2, \dots, \alpha_k)$ consists of nonzero real numbers, unless explicitly stated otherwise. Also, if two quantities, $X(n)$ and $Y(n)$, vary with respect to some natural number parameter, $n$, then we write $X(n) \lesssim Y(n)$ if there exist constants, $C$ and $N$, both independent of $n$, such that for all $n> N$, we have $X(n)\leq CY(n)$. If $X(n) \lesssim Y(n)$ and $Y(n) \lesssim X(n)$, we write $X(n) \approx Y(n).$

\subsection{Main results}

In \cite{BS}, Dan Barker and the third listed author gave the following bound on the number of hinges (2-chains) in a large finite point set in the plane. They go on to show that this bound is tight.

\begin{theorem}\label{hinges}
Given a large, finite set $E$ of $n$ points in $\mathbb R^2$, and a pair of nonzero real numbers $(\alpha_1, \alpha_2)$, the maximum number of triples of points, $(R_1, R_2, R_3)\in E^3$ such that $R_1\cdot R_2 = \alpha_1$ and $R_2\cdot R_3 = \alpha_2$ is no more than $\lesssim n^2.$
\end{theorem}

In this note, we continue the work in \cite{BS} by offering upper bounds on the number of times that a $k$-chain of a given type can occur in any large finite point set. In what follows, we will always assume that $k,$ the length of the chain, is constant with respect to $n,$ the total number of points in a given large, finite point set. We denote this by $k\lesssim 1.$

\begin{theorem}\label{2chainzgeneral}
Given a large, finite set $E$ of $n$ points in $\mathbb R^2$ and a natural number $k\lesssim 1$, the maximum number of $k$-chains of the type $(\alpha_1,\dots, \alpha_k)$ that can exist in $E$, is $\lesssim n^\frac{2(k+1)}{3}.$
\end{theorem}

The next result gives a lower bound for how many $k$-chains of a given type we can construct in $\mathbb R^2.$

\begin{proposition}\label{sharpChains}
There exists a set $E$ of $n$ points in $\mathbb{R}^2$ and a $k$-tuple of nonzero real numbers $(\alpha_1,\dots, \alpha_k)$, for which which there are at least $n^{\ceil{(k+1)/2}}$ instances of $k$-chains of the type $(\alpha_1,\dots, \alpha_k),$ for any $k\lesssim 1.$
\end{proposition}

This result is surprising because the corresponding estimates for chains of distances appeared to be the same from previous investigations of 1- and 2-chains, but for higher values of $k$, their behaviors are very different. In particular, the best known upper bound on the number of times a single distance (1-chain) can occur is $n^\frac{4}{3}.$ The upper bound on the number of times a particular distance hinge (2-chain) can occur is $n^2,$ which is sharp. For dot products, we have the same upper bounds on 1- and 2-chains. Moreover, both bounds are known to be sharp.

However, the similarities end there. Both upper and lower bounds on the number of $k$-chains are different for larger values of $k.$ This is addressed in further detail in Subsection \ref{starlikeProof}. For example, Proposition \ref{sharpChains} tells us that we can construct a large finite set of $n$ points for which there are $\gtrsim n^4$ occurrences of a particular dot product 6-chain, but the main result in \cite{PSS} (Theorem 1.1 in that paper), implies that there can be no more than $\lesssim n^\frac{1079}{300}\lesssim n^{3.6}$ occurrences of any distance 6-chain.

Note, we have only considered nonzero dot products here, because of point sets like the one in the following construction.
\begin{remark}\label{2Lenz}
For any $k\lesssim 1$, there are infinitely many $n$ for which we can get $\gtrsim n^{k+1}$ occurrences of $k$-chains of the type $(0, \dots,0)$ by putting $n/2$ points on the $x$-axis and $n/2$ points on the $y$-axis.
\end{remark}

\subsection{Special cases}

One key feature of the construction in the proof of Proposition \ref{sharpChains} is the fact that dot products can repeat in ways that distances cannot. However, with a restriction on the types of point sets under consideration, so that this overlap of dot products is forbidden, we can get much better results by slightly modifying the proof of the main two-dimensional result in \cite{FK} (Theorem 2 in that paper). We discuss this further below.
\newpage

\begin{corollary} \label{starlike}
Given a large, finite set $E$ of $n$ points in $\mathbb R^2,$ with the property that no two points of $E$ lie on the same line as the origin, any $\epsilon >0,$ and a natural number $k\lesssim 1$, the maximum number of $k$-chains of the type $(\alpha_1,\dots, \alpha_k)$ that can exist in $E$ is no more than
$$\lesssim
\begin{cases}
n^\frac{k+3}{3} \qquad & \text{ if } k \equiv 0 \Mod{3}, \\
n^{\frac{k+3}{3}+\epsilon} & \text{ if } k \equiv 1 \Mod{3}, \\
n^\frac{k+4}{3} & \text{ if } k \equiv 2 \Mod{3}.
\end{cases}$$
\end{corollary}

Next, we follow \cite{BS} in offering $k$-chain bounds for sets of points where no line has many points.

\begin{corollary}\label{lightLines}
Given a large, finite set $E$ of $n$ points in $\mathbb R^2$, with no more than $t$ points on any line and a natural number $k\lesssim 1$, the maximum number of $k$-chains of the type $(\alpha_1,\dots, \alpha_k)$ that can exist in $E$, is 
$$\lesssim \begin{cases}
\left(\log_2 n \right)^\frac{2k-6}{3} t^\frac{k-3}{3}n^{\frac{4k+12}{9}} \qquad& \text{ if } k \equiv 0 \Mod 3,\\
\left(\log_2 n \right)^\frac{2k-2}{3} t^\frac{k-1}{3}n^{\frac{4k+8}{9}} & \text{ if } k \equiv 1 \Mod 3,\\
\left(\log_2 n \right)^\frac{2k+2}{3} t^\frac{k+1}{3}n^{\frac{4k+4}{9}} & \text{ if } k\equiv 2 \Mod 3.
\end{cases}$$
\end{corollary}

We now introduce {\it $s$-adaptability}, which is used to quantify how well-distributed a set of points is. This property has been used to study many types of geometric point configuration problems. Using $s$-adaptability, results about discrete point sets can be partially translated to apply to sets with positive Hausdorff dimension. These latter results have consequences in geometric measure theory. See \cite{BIT, IJL, IMS, IRU, IS2}, for example.
A large, finite point set $E \subset [0,1]^2$, is said to be {\it $s$-adaptable} if the following two conditions hold:
\begin{align*}
\text{(energy)} \qquad &\frac{1}{{n \choose 2}}\sum_{\substack{P, Q \in E\\P\neq Q}} |P-Q|^{-s} \lesssim 1,\\
\text{(separation)} \qquad &\min\{|P-Q| : P,Q\in E,~ P\neq Q \}\geq n^{-\frac{1}{s}}.\\
\end{align*}
The separation condition from the definition of $s$-adaptability guarantees that there will be no more than $\lesssim n^{\frac{1}{s}}$ points on a line, so we can appeal to Corollary \ref{lightLines} and get the next result as a corollary.
\begin{corollary}\label{sAdapt}
Let $E \subset [0,1]^2$ be a set of $n$ points that is $s$-adaptable. For $s\leq 2$ and a natural number $k\lesssim 1$, the maximum number of $k$-chains of the type $(\alpha_1,\dots, \alpha_k)$ that can exist, is $$\lesssim \begin{cases}
\left(\log_2 n \right)^\frac{2k-6}{3} n^{\frac{4k+12}{9}+\frac{k-3}{3s}} \qquad& \text{ if } k \equiv 0 \Mod 3,\\
\left(\log_2 n \right)^\frac{2k-2}{3} n^{\frac{4k+8}{9}+\frac{k-1}{3s}} & \text{ if } k \equiv 1 \Mod 3,\\
\left(\log_2 n \right)^\frac{2k+2}{3} n^{\frac{4k+4}{9}+\frac{k+1}{3s}} & \text{ if } k\equiv 2 \Mod 3.
\end{cases}$$
\end{corollary}

We pause to note that no large, finite set of $n$ points in $[0,1]^2$ can be $s$-adaptable for $s>2,$ without violating the separation condition. To see this, partition the unit square into a $\sqrt n \times \sqrt n$ grid of squares of side-length $n^{-\frac{1}{2}}.$ Notice that there are $n$ such squares, and if points are to be separated by a distance much greater than $n^{-\frac{1}{2}}$, many of the small squares must be empty, but then there cannot be $n$ points in total. Moreover, when a set is not $s$-adaptable for any $s\geq \frac{3}{2},$ Corollary \ref{sAdapt} is outperformed by Theorem \ref{2chainzgeneral}. So the effective range for Corollary \ref{sAdapt} is $\frac{3}{2}\leq s \leq 2.$


\subsection{Higher dimensions}

Thus far, we have only considered point sets in the plane. Much like in Remark \ref{2Lenz}, there are point sets in higher dimensions that have many dot product chains of a given type. The difference here is that in higher dimensions, we can construct arbitrarily long dot product $k$-chains with $n^{k+1}$ points that have nonzero dot products, as opposed to in two dimensions, where we could only do that for zero dot products.

\begin{remark}\label{3Lenz}
Given a natural number $k \lesssim 1$ and any type of dot product $k$-chain, $(\alpha_1,\dots, \alpha_k),$ we can arrange about $n/(k+1)$ points along each of the following lines:
$$\{(x,y,z): x=1, y=0\},$$
$$\{(x,y,z): x=\alpha_1, z=0\},$$
$$\{(x,y,z): x=\alpha_2/\alpha_1, y=0\},$$
$$\{(x,y,z): x=\alpha_1\alpha_3/\alpha_2, z=0\}, \dots$$
and so on, so that the dot product of a point from the $j$th line and a point from the $(j+1)$th line is $\alpha_j,$ and alternating the free variable between $y$ and $z$ each time. This gives us a set of $n$ points with a total of about $n^{k+1}$ dot product $k$-chains of the specified type.
\end{remark}

Both Remark \ref{2Lenz} and Remark \ref{3Lenz} were inspired by the celebrated Lenz construction in $\mathbb R^4$ of $n/2$ points on a unit circle in the first two dimensions and $n/2$ points on a circle in the next two dimensions, which will give about $n^2$ point pairs (one point from each circle) separated by a distance of $\sqrt 2.$ These constructions make it clear that the study of chains in higher dimensions is trivial unless there are extra restrictions put on the point sets under consideration. For this reason, we offer the next results on point sets where we have some extra conditions on how many points can be on hyperplanes.

\begin{corollary}\label{hiDim}
Given any natural number $k \lesssim 1,$ and any large, finite set $E$ of $n$ points in $\mathbb R^d$, with no more than $r$ points on any $(d-1)$-hyperplane, and no more than $t$ points on any $(d-2)$-hyperplane, and any $\epsilon>0,$ the number of occurrences of a dot product $k$-chain of type $(\alpha_1,\dots, \alpha_k)$ is bounded above by
$$\lesssim \begin{cases}
n^\frac{k+3}{3}r^\frac{k-3}{3}t^2+n^{\frac{(4d-3)(k-1)+18d-8}{6d-3}+\epsilon}r^\frac{k-3}{3}t^\frac{2d-2}{2d-1} \qquad& \text{ if } k \equiv 0 \Mod 3,\\
 n^\frac{k+2}{3}r^\frac{k-1}{3}t+n^{\frac{(4d-3)(k-1)+9d-6}{6d-3}+\epsilon}r^\frac{k-1}{3}t^\frac{d-1}{2d-1} & \text{ if } k \equiv 1 \Mod 3,\\
n^\frac{k+1}{3}t^\frac{2k+2}{3}+n^{\frac{(4d-3)(k+1)}{6d-3}+\epsilon}t^{\frac{(2d-2)(k+1)}{6d-3}+\epsilon}+n^\frac{k+1}{3}r^\frac{k+1}{3} & \text{ if } k\equiv 2 \Mod 3.
\end{cases}$$
\end{corollary}

\subsection{Organization of this paper}
We will begin by recalling some of the elementary notions used herein in Section \ref{prelim}. Section \ref{proofs} follows, beginning with the most simple arguments. Since many of the arguments have the same basic structure, we omit details in the later proofs.

\section{Preliminaries}\label{prelim}

To keep track of dot products, we introduce some geometric tools.

\subsection{The $\alpha$-line for A}

Let $A$ be a point in $\arr^{2}$ with coordinates $(a_1,a_2)$, and let $\alpha \neq 0$ be a real number. A point $B$ in $\arr^{2}$ satisfies $A \cdot B = \alpha$, if and only if it lies on the line having the equation
$$a_{1}{x}+ a_{2}y=\alpha.$$
This line is called \textit{the $\alpha$-line for A,} denoted $\ell_\alpha(A).$ Note that $B$ is on the $\alpha$-line of $A$, if and only if $A$ is on the $\alpha$-line of $B$. We also call the unique line through a point $A$ and the origin the {\it radial line} of $A,$ and note that it is perpendicular to $\ell_\alpha(A)$ for any $\alpha.$\\

\begin{minipage}{5in}
\includegraphics[scale=1]{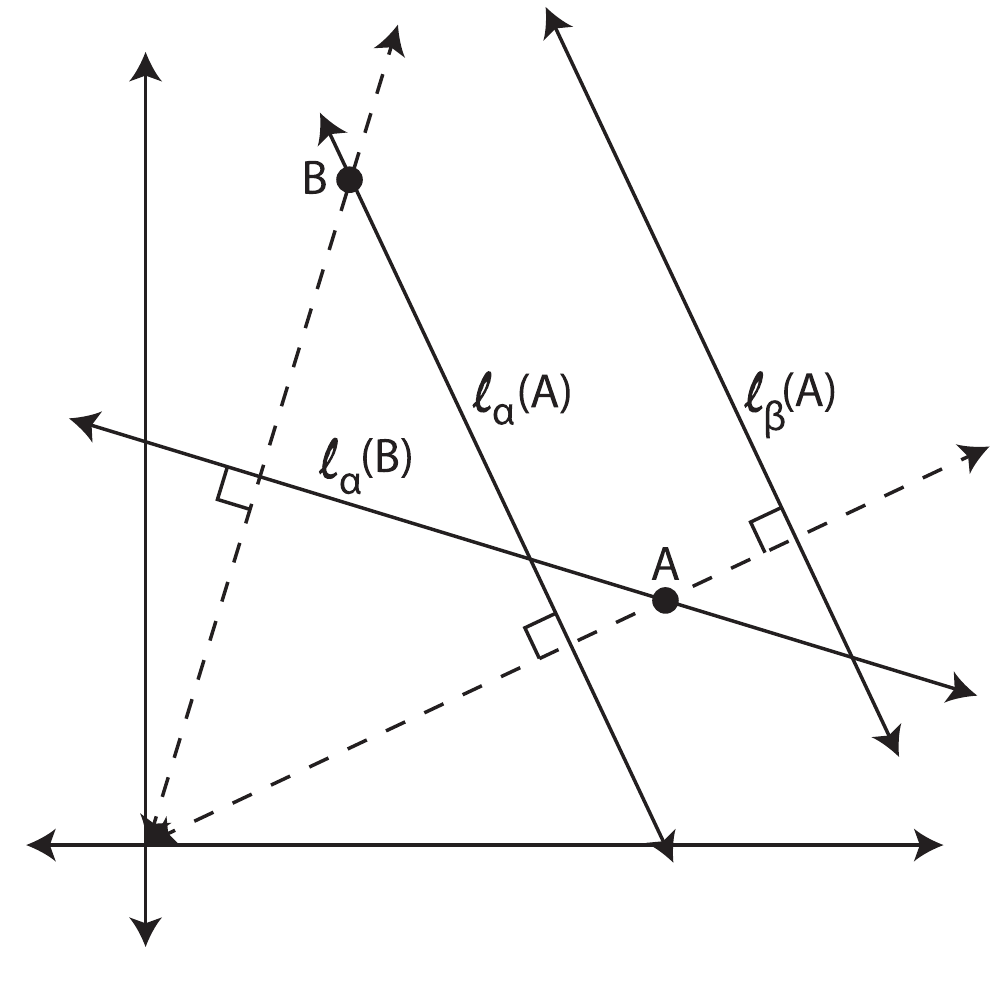}\label{dpChainsFig1}\\
{\bf Figure 1:} Here we have the points $A$ and $B$, neither of which are the origin, and two positive numbers $\alpha$ and $\beta.$ The dashed lines are the radial lines of $A$ and $B$, and their $\alpha$-lines are perpendicular to their respective radial lines. The points $A$ and $B$ have a dot product of $\alpha,$ so $A$ is on $\ell_\alpha(B),$ and $B$ is on $\ell_\alpha(A).$ Notice that for $\beta \neq \alpha$, the points that have dot product $\beta$ with the point $A$ are on another line, $\ell_\beta(A)$ that is parallel to $\ell_\alpha(A).$\\\\
\end{minipage}

\subsection{Basic tools}

The following lemma appears in \cite{BS} and will be fundamental to our results. We include the short proof as it shows the flavor of the arguments to come.

\begin{lemma}\label{singlePoint}
If $A$ and $C$ are two points in $\mathbb{R}^2$ that do not lie on the same radial line, and $\alpha, \beta \in \mathbb{R}\setminus\{0\},$ then there exists exactly one point $B \in\mathbb{R}^2$ such that
$$B \cdot A = \alpha \text{  and  } B \cdot C = \beta $$
\end{lemma}

\begin{proof} Let $\ell_\alpha(A)$ be the alpha line for $A$ and let $\ell_\beta(C)$ be the beta line for $C.$ Since $A$ and $C$ do not lie on the same radial line and $\ell_\alpha(A)$ and $\ell_\beta(C)$ are perpendicular to the radial lines of $A$ and $C$, respectively, they are not parallel to each other and hence intersect in exactly one point. Therefore, there exists exactly one point $B \in \mathbb R^2$ such that $B \cdot A= \alpha \text { and }  B \cdot C=\beta.$
\end{proof}

The celebrated Szemer\'edi-Trotter Theorem from \cite{ST83} will also come into play in what follows. 

\begin{theorem}\label{ST}
Given $n$ points and $m$ lines in the plane, the number of point-line pairs, such that the point lies on the line is
$$\lesssim\left(n^{\frac{2}{3}}m^{\frac{2}{3}}+n+m\right).$$
\end{theorem}

\section{Proofs}\label{proofs}

We begin with the following simple lemma, using Szemer\'edi-Trotter to bound how often a particular dot product can occur in a set of points in the plane.
\begin{lemma}\label{singleDP}
Given a large, finite set $E$ of $n$ points in the plane, no $\alpha\neq 0$ can be a dot product determined by pairs of points from $E$ more than $n^\frac{4}{3}$ times.
\end{lemma}
\begin{proof}
We can ignore the origin, as it has dot product zero with every point, so in what follows, we will assume that the origin is not in $E$. For each point $R\in E,$ draw the corresponding $\alpha$-line, $\ell_\alpha(R).$ So for each $R=(r_1,r_2)\in E,$ we have that $\ell_\alpha(R)$ can be written in the form $r_1x+r_2y=\alpha.$ We now show that these are unique.
\newpage

Consider any two distinct points $A, C\in E$. In the case that $A$ and $C$ do not lie on the same radial line, Lemma \ref{singlePoint} shows that there exists exactly one point, $B$, such that $A\cdot B = \alpha = B \cdot C.$ This means that $\ell_\alpha(A)\cap\ell_\alpha(C)$ is a single point, so they must be distinct lines.

Now suppose that $A=(a_1,a_2)$ and $C=(c_1,c_2)$ lie on the same radial line, and that it is not the $y$-axis. Then $a_1$ and $c_1$ are both nonzero, and there exists a nonzero $\lambda \in \mathbb R$ such that
$$c_1=\lambda a_1, \text{ and } c_2 = \lambda a_2.$$
By way of contradiction, suppose that $\ell_\alpha(A)=\ell_\alpha(C).$ Then we would have
$$a_1x+a_2y=\alpha=c_1x+c_2y=\lambda a_1x+\lambda a_2y,$$
for all $x\in \mathbb R.$ This implies that $a_1=\lambda a_1,$ which in turn tells us that $\lambda=1.$ But this means that $A$ and $C$ somehow share the same first coordinate and lie on the same radial line. The only radial line on which distinct points share the first coordinate is the $y$-axis, but we have assumed that $A$ and $C$ do not lie on the $y$-axis, so in this case $\ell_\alpha(A)$ cannot completely coincide with $\ell_\alpha(C)$.

In the case that $A$ and $C$ both lie on the $y$-axis, we would have that $a_1=c_1=0$, and the definition of $\ell_\alpha(A)$ would tell us $0\cdot x + a_2y=\alpha$, or that we can write $\ell_\alpha(A)$ as $y=(\alpha/a_2).$ Arguing similarly for $C$, we would see that $\ell_\alpha(C)$ could be written as $y=(\alpha/c_2)$. Putting these together yields
$$\frac{\alpha}{a_2} = y = \frac{\alpha}{c_2}.$$
But this means that $c_2 = a_2.$ This means that both points have coordinates $(0,a_2),$ meaning that they are the same point, which contradicts the fact that they are distinct. So again, $\ell_\alpha(A)$ cannot completely coincide with $\ell_\alpha(C)$.

Finally, we apply Szemer\'edi-Trotter. Since each of the $n$ points is associated to a unique line, we have a set of $n$ lines. Notice that by definition, if a point $P$ lies on a line $\ell_\alpha(Q),$ for some point $Q$, then $P\cdot Q=\alpha.$ That is to say, the number of point pairs for $E$ that determine the dot product $\alpha$ is precisely the number of incidences of points from $E$ and lines of the form $\ell_\alpha(Q),$ for some $Q\in E.$ Since there are $n$ points and $n$ lines, Theorem \ref{ST} tells us that the number of point-line incidences, and therefore occurrences of the dot product $\alpha$, is $\lesssim (n\cdot n)^\frac{2}{3}+n+n \approx n^\frac{4}{3},$ as claimed.
\end{proof}
\subsection{Proof of Theorem \ref{2chainzgeneral}}

\begin{proof} Let $E \subset \mathbb R^2$ be any set of $n$ points, and consider a given $k$-chain type $(\alpha_1, \dots, \alpha_k)$. Our aim is to bound the number of $(k+1)$-tuples of points in $E$ that determine these dot products pairwise, suppressing dependence on the $k$-tuples $(\alpha_1, \dots, \alpha_k)$, that is, the size of $\mathcal C_k(E), $ where
$$\mathcal C_k(E):= \{(R_1, \dots, R_{k+1})\in E^{k+1}:R_i\cdot R_{i+1}=\alpha_i, i=1,\dots, k\}.$$
We break into cases depending on the congruence class of $k$ modulo 3. From there, each case is broken up into component pieces which we estimate with repeated applications of Theorem \ref{hinges} and Lemma \ref{singleDP}.

\subsubsection{Case 1: $k\equiv 2 \Mod 3$} First off, if $k=2$, then we have the desired bound by direct appeal to Theorem \ref{hinges}. If $k>2$, then we just apply it repeatedly. Specifically, there exists a nonnegative integer $j$ such that $k=3(j-1)+2.$ Since $k+1=(3(j-1)+2)+1=3j,$ we are dealing with $(3j)$-tuples of points. We will break each $(3j)$-tuple down into $j$ consecutive triples. Start with $(R_1, R_2, R_3).$ By Theorem \ref{hinges}, we know that there can be no more than $n^2$ triples of the form $(R_1, R_2, R_3)$ with $R_1\cdot R_2 = \alpha_1$ and $R_2\cdot R_3 =\alpha_2.$ We will ignore any possible relationship between $R_3$ and $R_4$. Appealing to Theorem \ref{hinges} again, there can be no more than $n^2$ triples of the form $(R_4, R_5, R_6)$ with $R_4\cdot R_5 = \alpha_4$ and $R_5\cdot R_6 =\alpha_5.$ We continue in this fashion to bound the number of triples of each subsequent type. In so doing, we get a bound on the set
$$\mathcal C_{k,2}(E)=\{(R_1, \dots, R_{k+1})\in E^{k+1}:R_i\cdot R_{i+1}=\alpha_i, i=1,2,4,5, \dots, k,~ 3\hskip-1ex\not| i \}.$$
Namely, we can see that $|\mathcal C_{k,2}(E)|\lesssim (n^2)^{j},$ because for each of the $j$ triples, we get a bound of $n^2.$ Notice that $\mathcal C_k(E)\subset \mathcal C_{k,2}(E)$, so we get $|\mathcal C_k(E)|\leq |\mathcal C_{k,2}(E)|\lesssim (n^2)^{j}=n^\frac{2(k+1)}{3}.$

\subsubsection{Case 2: $k\equiv 1 \Mod 3$} In this case, there exists a nonnegative integer $j$ such that $k= 3j+1.$ So we can write $k+1 = 3j + 2.$ This means that we are looking at $j$ triples of points, followed by a pair of points. We can handle the $j$ triples by Case 1, and the final pair of points separately. That is to say, since $k-2\equiv 2 \Mod 3,$ the number of $(k-2)$-chains will be bounded above by
$$|\mathcal C_{k-2}(E)|\lesssim \left(n^2\right)^j=n^\frac{2(k-1)}{3}.$$
Now, we will ignore the relationship between $R_{k-1}$ and $R_k,$ and use Lemma \ref{singleDP} to bound the number of point pairs $(R_k, R_{k+1})$ such that $R_k\cdot R_{k+1}=\alpha_k$ by $n^\frac{4}{3}.$ That is to say,
$$|\mathcal C_1(E)|\lesssim n^\frac{4}{3}.$$
Putting these together, we get
$$|\mathcal C_k(E)|\leq |\mathcal C_{k-2}(E)|\cdot|\mathcal C_1(E)|\lesssim n^\frac{2(k-1)}{3}\cdot n^\frac{4}{3} = n^\frac{2(k+1)}{3}.$$

\subsubsection{Case 3: $k\equiv 0 \Mod 3$} In this case, there exists a nonnegative integer $j$ such that $k= 3j.$ Therefore we can write $k+1 = 3(j-1) + 2+2.$ So we have $(j-1)$ triples, followed by two pairs. So, similar to the previous case, we will deal with the $(j-1)$ triples using Theorem \ref{hinges}, then apply Lemma \ref{singleDP} twice to get
$$|\mathcal C_k(E)|\leq |\mathcal C_{k-4}(E)|\cdot|\mathcal C_1(E)|\cdot|\mathcal C_1(E)|\lesssim n^\frac{2(k-3)}{3}\cdot n^\frac{4}{3}\cdot n^\frac{4}{3} = n^\frac{2(k+1)}{3}.$$
\end{proof}

It is worth pointing out that in Case 1, if $k$ is also congruent to 5 modulo 6, we can get the same bounds by breaking up every sextuple into three pairs and using Lemma \ref{singleDP}, instead of breaking them into two triples and using Theorem \ref{hinges}.

\subsection{Proof of Proposition \ref{sharpChains}} 
We will work inductively using a process similar to the proof of Theorem 2 of \cite{BS}. Given a natural number $k\geq 2,$ we will select a $k$-tuple of dot products $(\alpha_1, \dots, \alpha_k)$ and construct a set of points that exhibits many $k$-chains of that type. We note that for this construction, we will have the restriction that $\alpha_{2j}=\alpha_{2j-1}$ for all $j=1,\dots, \lfloor k/2\rfloor.$
\begin{proof}
Select a point $(1,1)$ and call it $R_2.$ Now choose some nonzero $\alpha_1\in\mathbb R.$ Place $n-k$ points on the $\alpha_1$-line $\ell_{\alpha_1}(R_2)$. We now set $\alpha_2=\alpha_1.$ Now we have $\approx n$ choices for each of the points $R_1$ and $R_3$ so that $R_1\cdot R_2 = \alpha_1$ and $R_2\cdot R_3 = \alpha_2=\alpha_1.$ Thus, the bound holds for $k = 2$. 

Now select the point $(2,2)$, and call it $R_4.$ Let $\alpha_3$ be the unique nonzero dot product so that $\ell_{\alpha_1}(R_2)$ and $\ell_{\alpha_3}(R_4)$ are coincident. As before, we now have $\approx n$ choices for $R_5$ so that $R_3 \cdot R_4=\alpha_3$ and $R_4 \cdot R_5=\alpha_4=\alpha_3.$

Next, if necessary, we select the point $(3,3)$ and call it $R_6.$ Let $\alpha_5$ be the unique nonzero dot product so that $\ell_{\alpha_3}(R_4)$ and $\ell_{\alpha_5}(R_6)$ are coincident. Let $m=\lceil k/2 \rceil.$ We continue to repeat this process until we have selected the point $(m, m)$ to be $R_k$, if $k$ is even, or $R_{k+1}$, if $k$ is odd. Now, each of the even indexed points $R_j$ are fixed, but the odd indexed points each have $\approx n$ choices. Since there are $m$ odd indices, we have a total of $\approx n^m$ occurrences of the $k$-chain of type $(\alpha_1, \alpha_1, \alpha_2, \alpha_2, \dots, \alpha_m),$ as claimed.
\end{proof}

The idea behind this construction is actually a bit more flexible than written in the proof. The basic idea is to put $\lceil k/2\rceil$ points on the same radial line, and to evenly distribute the rest of them on a family of about $\lfloor k/2\rfloor$ lines perpendicular to the original line. In the construction given above, we put this family of lines all on one line, which restricts the possible values of dot products, $\alpha_j$, slightly more than is necessary.

\subsection{Proof of Corollary \ref{starlike}}\label{starlikeProof}

Given distinct points $P$ and $Q$, the circle of radius $\alpha$ centered at $P$ can intersect the circle of radius $\beta$ centered at $Q$ in at most two points. So the number of points that are of distance $\alpha$ to $P$ and distance $\beta$ to $Q$ is at most two. However, the corresponding property does not necessarily hold for dot products. Namely, there exist distinct points $P$ and $Q$ with infinitely many points that are of dot product $\alpha$ to $P$ and $\beta$ to $Q,$ even with $\alpha, \beta \neq 0.$ This happens when $P$ and $Q$ lie on the same radial line, as the next lemma shows.

\begin{lemma}\label{sameRadial}
If two points, $P$ and $Q$ have $\ell_\alpha(P)=\ell_\beta(Q),$ then $P$ and $Q$ lie on the same radial line.
\end{lemma}
\begin{proof}
Notice that if $P=(p_1,p_2)$ and $Q=(q_1,q_2)$ are the same point, the conclusion is trivially true. So we now assume that $P\neq Q.$ Let $\ell_\alpha(P)$ have the equation $y=mx+b.$ For any $x\in \mathbb R,$ we must have that $p_1x+p_2y=\alpha.$ Solving this for $y$ yields that
$$y=-\frac{p_1}{p_2}x+\frac{\alpha}{p_2}.$$
So we can see that $m=-\frac{p_1}{p_2}$ and $b=\frac{\alpha}{p_2}.$ But  $\ell_\beta(Q)$ is the same line, with the same equation, so we can similarly argue that $m=-\frac{q_1}{q_2}$ and $b=\frac{\beta}{q_2}.$ Comparing the two expressions for $m$ gives us
$$-\frac{p_1}{p_2} = m = -\frac{q_1}{q_2},$$
which implies that
\begin{equation}\label{sameM}
\frac{p_2}{p_1} = \frac{q_2}{q_1}.
\end{equation}
Notice that $P$ lies on a line through the origin with slope $\frac{p_2}{p_1},$ and $Q$ lies on a line through the origin with slope $\frac{q_2}{q_1}.$ But \eqref{sameM} says that these are the same line through the origin.
\end{proof}

We now state the main two-dimensional estimate in \cite{FK}, by Frankl and Kupavskii, which is Theorem 2 in that paper (an improvement on similar estimates from \cite{PSS}). We refer to a {\it distance $k$-chain} as a $k$-chain defined by distances instead of dot products. Since distances between distinct points are strictly positive, we only concern ourselves with $\alpha_j>0$ when dealing with distances. Let $u_2(n)$ denote the maximum number of pairs of points separated by exactly a unit distance in any set of $n$ points in the plane.

\begin{theorem}\label{FKthm}[Theorem 2 in \cite{FK}]
Given a large, finite set $E$ of $n$ points in $\mathbb R^2,$ any $\epsilon >0,$ and a natural number $k\lesssim 1$, the maximum number of distance $k$-chains of the type $(\alpha_1,\dots, \alpha_k)$ that can exist in $E$ is no more than
$$\lesssim
\begin{cases}
n^\frac{k+3}{3} \qquad & \text{ if } k \equiv 0 \Mod{3}, \\
n^{\frac{k-1}{3}+\epsilon}u_2(n) & \text{ if } k \equiv 1 \Mod{3}, \\
n^\frac{k+4}{3} & \text{ if } k \equiv 2 \Mod{3}.
\end{cases}$$
\end{theorem}

The key to the proof of Theorem \ref{FKthm} is a bound on incidences between points and circles, and it heavily relies on the fact that circles of the form $\mathcal C_\alpha(P)$ and $\mathcal C_\beta(Q)$ can intersect in at most two points. In order to follow the proof through for dot product $k$-chains, we would need to replace the circles that come from distances with lines that come from dot products. Now, in general, we would not have the guarantee that $\ell_\alpha(P)$ and $\ell_\beta(Q)$ intersect in a small number of points. However, with the additional hypothesis that no two points lie on the same radial line, we are guaranteed by Lemma \ref{sameRadial} that $\ell_\alpha(P)$ and $\ell_\beta(Q)$ intersect in at most one point, and we can follow the proof through in the cases that $k\equiv 0,2 \Mod{3},$ replacing circles with lines. To handle the case that $k\equiv 1 \Mod{3},$ we notice that by applying Lemma \ref{singleDP} for the maximum number of occurrences of a single dot product, we can replace $u_2(n)$ with $n^\frac{4}{3}.$

\subsection{Proof of Corollary \ref{lightLines}}

The following result is a rephrasing of what is proved in \cite{BS} (Theorem 2 and Remark 3 in that paper). It says that if we have some bounds on the number of points on a line, we can get better hinge bounds than the general case, as in Lemma \ref{hinges}.

\begin{theorem}\label{BSep}[Barker and Senger]
Given a large, finite set $E$ of $n$ points in $\mathbb R^2$, with no more than $t$ points on any line, the maximum number of hinges of the type $(\alpha_1, \alpha_2)$ that can exist in $E$, is $\lesssim \left(\log_2 n \right)^2 tn^{\frac{4}{3}}.$
\end{theorem}

To prove Corollary \ref{lightLines}, we repeat the proof of Theorem \ref{2chainzgeneral}, but use Theorem \ref{BSep} in place of Lemma \ref{hinges}. As before, we separate into three cases, depending on the congruence classes of $k$ modulo 3. Here we merely note the necessary modifications of each case of the proof of Theorem \ref{2chainzgeneral}.

\subsubsection{Case 1: $k \equiv 2 \Mod 3$} Find $j$ so that $k+1=3j$, and apply Theorem \ref{BSep} $j$ times to get an upper bound of
$$\lesssim \left(\left(\log_2 n \right)^2 tn^{\frac{4}{3}}\right)^j \approx \left(\log_2 n \right)^\frac{2k+2}{3} t^\frac{k+1}{3}n^{\frac{4k+4}{9}}.$$

\subsubsection{Case 2: $k\equiv 1 \Mod 3$} Set $j$ so that $k+1=3j+2$. We apply Theorem \ref{BSep} $j$ times and Lemma \ref{singleDP} once to get an upper bound of
$$\lesssim \left(\left(\log_2 n \right)^2 tn^{\frac{4}{3}}\right)^j\cdot n^\frac{4}{3} \approx \left(\log_2 n \right)^\frac{2k-2}{3} t^\frac{k-1}{3}n^{\frac{4k+8}{9}}.$$

\subsubsection{Case 3: $k\equiv 0 \Mod 3$} Fix $j$ so that $k+1=3(j-1)+2+2$. We apply Theorem \ref{BSep} $(j-1)$ times and Lemma \ref{singleDP} twice to get an upper bound of
$$\lesssim \left(\left(\log_2 n \right)^2 tn^{\frac{4}{3}}\right)^{(j-1)}\cdot n^\frac{4}{3}\cdot n^\frac{4}{3} \approx \left(\log_2 n \right)^\frac{2k-6}{3} t^\frac{k-3}{3}n^{\frac{4k+12}{9}}.$$

\subsection{Proof of Corollary \ref{hiDim}}

In order to prove Corollary \ref{hiDim}, we mimic the proof of Corollary \ref{lightLines}, but appeal to higher dimensional incidence theorems. As the construction in Remark \ref{3Lenz} shows, we have no hope to prove nontrivial estimates in higher dimensions in general, so we restrict ourselves to the case of sets where we have some control on the distribution of points. To this end, we introduce two estimates. The first is due to Ben Lund, in \cite{Lund}.

\begin{theorem}\label{lundHinges}[Lund, in \cite{Lund}]
Given a large, finite set $E$ of $n$ points in $\mathbb R^d$, with no more than $r$ points on any $(d-1)$-hyperplane, and no more than $t$ points on any $(d-2)$-hyperplane, a pair of nonzero real numbers $(\alpha_1, \alpha_2)$, and any $\epsilon>0,$ the maximum number of triples of points, $(R_1, R_2, R_3)\in E^3$ such that $R_1\cdot R_2 = \alpha_1$ and $R_2\cdot R_3 = \alpha_2$ is no more than
$$\lesssim nt^2+n^{\frac{4d-3}{2d-1}+\epsilon}t^{\frac{2d-2}{2d-1}+\epsilon}+nr.$$
\end{theorem}

This bound is an application of the following result, due to Lund, Sheffer, and de Zeeuw, from \cite{LSdZ}, but is based on the work of Fox, Pach, Suk, Sheffer, and Zahl, from \cite{FPSSZ}.

\begin{theorem}\label{dUDP}[Lund, Sheffer, and de Zeeuw, from \cite{LSdZ}]
Given a large, finite set $E$ of $n$ points and $m$ $(d-1)$-hyperplanes in $\mathbb R^d$, with no more than $t$ points on any pair of hyperplanes, and any $\epsilon>0,$ the maximum number of incidences of points and hyperplanes is no more than
$$\lesssim mt+m^{\frac{2(d-1)}{2d-1}+\epsilon}n^\frac{d}{2d-1}t^{\frac{d-1}{2d-1}}+n.$$
\end{theorem}

Though they are defined slightly differently in the last two results, for our purposes, both references to the quantity $t$ will coincide. Now we are ready to proceed. As before, we separate into three cases, depending on the congruence classes of $k$ modulo 3. Again, we merely note the necessary modifications of each case of the proof of Theorem \ref{2chainzgeneral}.

\subsubsection{Case 1: $k \equiv 2 \Mod 3$} Find $j$ so that $k+1=3j$, and apply Theorem \ref{lundHinges} $j$ times, setting $\epsilon' = \epsilon^\frac{3}{k+1}$. This yields an upper bound of
$$\lesssim \left(nt^2+n^{\frac{4d-3}{2d-1}+\epsilon'}t^{\frac{2d-2}{2d-1}+\epsilon'}+nr\right)^j$$ 
$$\approx n^\frac{k+1}{3}t^\frac{2k+2}{3}+n^{\frac{(4d-3)(k+1)}{6d-3}+\epsilon}t^{\frac{(2d-2)(k+1)}{6d-3}+\epsilon}+n^\frac{k+1}{3}r^\frac{k+1}{3}.$$

\subsubsection{Case 2: $k\equiv 1 \Mod 3$} Set $j$ so that $k+1=3j+2$. We apply Theorem \ref{lundHinges} $j$ times, each time with $\epsilon' = \epsilon^\frac{3}{k-1}$. We also apply Lemma \ref{dUDP}, with $m = n$ and $\epsilon''=\epsilon-\epsilon',$ since each point generates a unique hyperplane, akin to the $\alpha$-lines before. Notice that by definition, $t \leq r.$ This yields an upper bound of
$$\lesssim \left(nt^2+n^{\frac{4d-3}{2d-1}+\epsilon'}t^{\frac{2d-2}{2d-1}+\epsilon'}+nr\right)^j\cdot \left(nt+n^{\frac{2(d-1)}{2d-1}+\epsilon''}n^\frac{d}{2d-1}t^{\frac{d-1}{2d-1}}+n \right)$$
$$\approx n^\frac{k+2}{3}r^\frac{k-1}{3}t+n^{\frac{(4d-3)(k-1)+9d-6}{6d-3}+\epsilon}r^\frac{k-1}{3}t^\frac{d-1}{2d-1},$$
where we have omitted some cross terms, as they will always be dominated by terms present in the final expression given.

\subsubsection{Case 3: $k\equiv 0 \Mod 3$} Fix $j$ so that $k+1=3(j-1)+2+2$. Similarly to the previous case, we apply Theorem \ref{lundHinges} $(j-1)$ times and Lemma \ref{dUDP} twice, as in the previous case, with appropriate values of $\epsilon', \epsilon'', m, r,$ and $t$, to get an upper bound of 
$$\lesssim \left(nt^2+n^{\frac{4d-3}{2d-1}+\epsilon'}t^{\frac{2d-2}{2d-1}+\epsilon'}+nr\right)^{j-1}\cdot \left(nt+n^{\frac{2(d-1)}{2d-1}+\epsilon''}n^\frac{d}{2d-1}t^{\frac{d-1}{2d-1}}+n \right)^2$$
$$\approx n^\frac{k+3}{3}r^\frac{k-3}{3}t^2+n^{\frac{(4d-3)(k-1)+18d-8}{6d-3}+\epsilon}r^\frac{k-3}{3}t^\frac{2d-2}{2d-1},$$
where we have again omitted some cross terms.

\end{document}